\documentclass{article}
\usepackage[margin=20mm]{geometry}

\usepackage{amsmath, amssymb}
\usepackage{graphicx, wrapfig} 
\usepackage{amsthm}
\usepackage{url}
\usepackage{color}
\usepackage{mathrsfs}
\usepackage{bm}

\theoremstyle{plain}
\newtheorem{theorem}{Theorem}

\theoremstyle{definition}
\newtheorem{definition}{Definition}
\newtheorem{example}{Example}

\newtheorem{proposition}{Proposition}
\newtheorem{lemma}{Lemma}
\newtheorem{corollary}{Corollary}

\usepackage{color}

\title{The crossing matrix and the Burau matrix of pure braids}
\author{Ayaka Shimizu\thanks{Center for Soft Matter Physics, Ochanomizu University, 2-1-1, Otsuka, Bunkyo-ku, Tokyo, 112-8610, Japan. Email: shimizu.ayaka@ocha.ac.jp, shimizu1984@gmail.com} 
and Yoshiro Yaguchi\thanks{Maebashi Institute of Technology, 460-1, Kamisadori, Maebashi, Gunma, 371-0816, Japan. Email: y.yaguchi@maebashi-it.ac.jp}}
\date{\today}

\begin{document}

\maketitle

\begin{abstract}
We show that the first derivative of the Burau matrix at $t=1$ coincides with the Laplacian of the crossing matrix for pure braids. 
We also discuss conjugacy invariants derived from these matrices. 
\end{abstract}

\section{Introduction}
\label{section-intro}

The Burau matrix \footnote{
It is known that the reduced Burau representation is faithful for $m \leq 3$ (\cite{MP}) and unfaithful for $m\geq 5$ (\cite{JAM, LP, SB}). 
Very recently, Bharathram, Birman, and Brendle proved that it is faithful for $m=4$ (\cite{BBB}).}
$f(b)$ of an $m$-braid $b$ is an $m \times m$ matrix-valued invariant with Laurent polynomial entries that can be computed from the Artin generators (\cite{WB}; see also \cite{JSB}). 
Given the Burau matrix $f(b)=[f_{ij}(b;t)]$ of a braid $b$, we define the differentiated Burau matrix as follows. 
\begin{align*}
\frac{d}{dt}f(b)=\left[ \frac{d}{dt} f_{i j} (b;t) \right].
\end{align*}

\noindent In \cite{S}, Stoimenow constructed conjugacy invariants from the traces of the second and third derivatives of the Burau matrix at $t=1$, and used them to study conjugacy relations among braids obtained by iterated exchange moves.
In this paper, we study the first derivative of the Burau matrix at $t=1$ to find the relationship to the {\it crossing matrix}. 
The crossing matrix $C(b)$ of a braid $b$ is an $m \times m$ matrix-valued invariant with integer entries introduced in \cite{Bu} (see Section \ref{section-CM}), and it can be computed from a braid diagram. 
Let $L_{C(b)}$ be the Laplacian (see Section \ref{section-CM}) of the crossing matrix $C(b)$. 
In this paper, we prove the following theorem. 

\medskip 
\begin{theorem}
For any pure braid $b$, we have 
\begin{align*}
\left.\frac{d}{dt}f(b) \right|_{t=1}=L_{C(b)}.
\end{align*}
\label{thm-main}
\end{theorem}
\medskip 

\noindent As a consequence, the crossing matrix of a pure braid can be recovered from the Burau matrix\footnote{
Recently, Y. Kuno and the second author established an equivalence between the crossing matrix and the extended first Johnson homomorphism from a braid group in \cite{KY}. 
Therefore Theorem \ref{thm-main} also relates the first derivative of the Burau representation at $t=1$ to the extended first Johnson homomorphism from a braid group. 
} as follows. 

\medskip 
\begin{corollary}
For any pure braid $b$, the crossing matrix $C(b)=[c_{i j}]$ is obtained from the Burau matrix $f(b)=[f_{ij}(b;t)]$ as follows. 
\begin{align*}
c_{i j}= 
\begin{cases}
0 & (\text{if } i=j) \\
\left. -\frac{d}{dt}f_{i j}(b;t) \right|_{t=1} & ( \text{if } i\neq j)
\end{cases}
\end{align*}
\label{cor-cij}
\end{corollary}
\medskip 

\noindent We obtain the following corollary from Theorem \ref{thm-main} with a property of the crossing matrix of pure braids (Lemma \ref{lem-CM-properties}). 

\medskip 
\begin{corollary}
Let $b$ be a pure braid. 
Then $\left.\frac{d}{dt}f(b) \right|_{t=1}$ is a symmetric matrix. 
\label{cor-symmetric}
\end{corollary}
\medskip

The rest of this paper is organized as follows. 
In Section \ref{section-braid}, we recall basic properties of braids and pure braids. 
In Section \ref{section-BR}, we review the Burau representation and study its differentiated version. 
In Section \ref{section-CM}, we review the crossing matrix of braids. 
In Section \ref{section-pf}, we prove the main theorem and corollaries. 
In Section \ref{section-conjugacy}, we discuss conjugacy invariants. 
ChatGPT was used as an exploratory tool that led the authors to Theorem \ref{thm-main}. 
The theorem and its proof were subsequently established and independently verified by the authors.

\section{Braids and pure braids}
\label{section-braid}

An {\it $m$-braid} consists of $m$ mutually disjoint strands running monotonically from $m$ prescribed endpoints on the upper bar to $m$ prescribed endpoints on the lower bar. 
A regular projection of an $m$-braid onto a plane, together with over/under information at each crossing, is called an {\it $m$-braid diagram}. 
Each braid diagram can be represented by a word in the generators $\sigma_i^{\pm 1}$, where $\sigma_i^{\pm 1}$ are illustrated in Figure \ref{fig-s}. 
\begin{figure}[ht]
\centering
\includegraphics[width=6cm]{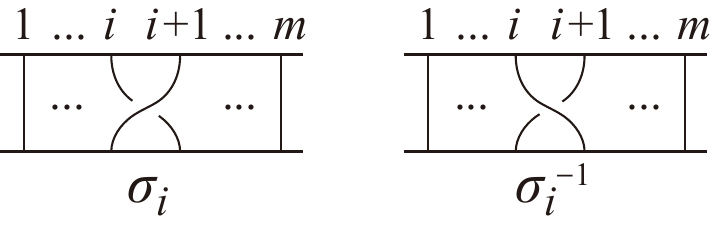}
\caption{The Artin generators $\sigma_i$ and $\sigma_i^{-1}$. }
\label{fig-s}
\end{figure}
It is well known that two $m$-braid diagrams represent the same braid if and only if they are related by a finite sequence of the following three types of transformations (\cite{Artin-1, Artin-2}). 
\begin{itemize}
\item[(1)] $\sigma_i^{\varepsilon} \sigma_i^{- \varepsilon}=1$ \ ($\varepsilon \in \{ \pm 1 \}$, $i \in \{ 1, 2, \dots , m-1 \}$), 
\item[(2)] $\sigma_i \sigma_{i+1} \sigma_i = \sigma_{i+1} \sigma_i \sigma_{i+1}$ \ ($i \in \{ 1, 2, \dots , m-2 \}$),
\item[(3)] $\sigma_i\sigma_j=\sigma_j \sigma_i$ \ ($i, j \in \{ 1, 2, \dots , m-1 \}$, $|j-i|>1$). 
\end{itemize}

\noindent These relations give a presentation of the braid group, denoted by $B_m$. 
Namely, 
\begin{align*}
B_m= \Bigg\langle \sigma_1, \sigma_2, \dots , \sigma_{m-1} \ \Bigg| \ 
\begin{matrix}\sigma_i\sigma_{i+1}\sigma_i=\sigma_{i+1}\sigma_i\sigma_{i+1} & (1 \leq i \leq m-2), \  \\
\sigma_i\sigma_j=\sigma_j\sigma_i & (|i-j|>1)
\end{matrix}
\Bigg\rangle .
\end{align*}

For an $m$-braid $b$, we call the strand that has the upper endpoint in the $i^{th}$ position from the left on the upper bar the {\it $i^{th}$ strand of $b$}. 
Each braid $b$ induces a permutation on $\{ 1, 2, \dots , m \}$ as follows. 
If the $i^{th}$ strand of $b$ has the lower endpoint at the $j^{th}$ position from the left on the lower bar, we set $\pi(i)=j$. 
The resulting permutation $\pi (1, 2, \dots , m)=(\pi(1), \pi(2), \dots , \pi(m))$ is called the {\it braid permutation of $b$}. 
A {\it pure braid} is a braid whose braid permutation is the identity. 
Let $P_m$ denote the set of pure $m$-braids. 
Then $P_m$ is a subgroup of $B_m$, called the  pure braid group on $m$ strands. 
Let 
\begin{align*}
A_{i j}= \sigma_{j-1} \sigma_{j-2} \dots \sigma_{i+1} \sigma_i^2 \sigma_{i+1}^{-1} \dots \sigma_{j-2}^{-1} \sigma_{j-1}^{-1} \in P_m
\end{align*}
($1 \leq i<j \leq m$) as shown in Figure \ref{fig-hook}.  
When $j=i+1$, we assume that $A_{i \ i+1}= \sigma_i^2$. 
We call $A_{i j}$ a {\it hook generator}. 
\begin{figure}[ht]
\centering
\includegraphics[width=5cm]{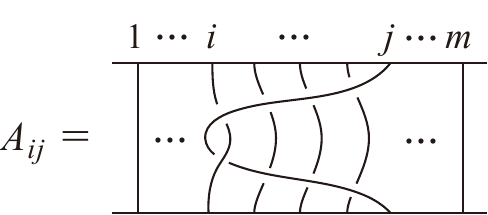}
\caption{A hook generator $A_{i j}=\sigma_{j-1} \sigma_{j-2} \dots \sigma_{i+1} \sigma_i^2 \sigma_{i+1}^{-1} \dots \sigma_{j-2}^{-1} \sigma_{j-1}^{-1}$. }
\label{fig-hook}
\end{figure}
The following lemma is well known (see, e.g., \cite{JSB}). 

\medskip 
\begin{lemma}
The pure braid group $P_m$ is generated by $A_{i j}$ $(1 \leq i<j \leq m)$. 
\end{lemma}
\medskip

\section{Burau representation}
\label{section-BR}

\subsection{Burau representation}

Let $\Lambda = \mathbb{Z}[ t, t^{-1}]$ be the ring of Laurent polynomials in $t$. 
Let $GL_m( \Lambda)$ denote the set of $m \times m$ invertible matrices over $\Lambda$. 
Then $GL_m(\Lambda)$ is a group under matrix multiplication. 
The {\it (unreduced) Burau representation}\footnote{
There are two versions of Burau representation: the unreduced and reduced Burau representations. 
In this paper, we refer to the unreduced version simply as the Burau representation. 
}
is a group homomorphism $f : B_m \to GL_m(\Lambda)$ defined on the Artin generators by 
\begin{align*}
f ( \sigma_i)= I_{i-1} \oplus 
\begin{bmatrix}
1-t & 1 \\
t & 0
\end{bmatrix}
\oplus I_{m-i-1} \ (1 \leq i \leq m-1),  
\end{align*}
where $I_k$ denotes the $k \times k$ identity matrix and $I_0$ denotes the empty matrix. 
The inverse of $f (\sigma_i)$ is given by 
\begin{align*}
f(\sigma_i)^{-1}=f ( \sigma_i^{-1})= I_{i-1} \oplus 
\begin{bmatrix}
0 & t^{-1} \\
1 & 1-t^{-1}
\end{bmatrix}
\oplus I_{m-i-1}.  
\end{align*}

\medskip 
\begin{example}
For $\sigma_1\sigma_2^{-1} \in B_3$, we have 
\begin{align*}
f(\sigma_1\sigma_2^{-1})=f(\sigma_1)f(\sigma_2)^{-1}=
\begin{bmatrix}
1-t & 1 & 0 \\
t & 0 & 0 \\
0 & 0 & 1 
\end{bmatrix}
\begin{bmatrix}
1 & 0 & 0 \\
0 & 0 & t^{-1} \\
0 & 1 & 1-t^{-1} 
\end{bmatrix}
=
\begin{bmatrix}
1-t & 0 & t^{-1} \\
t & 0 & 0 \\
0 & 1 & 1-t^{-1}
\end{bmatrix}.
\end{align*}
\end{example}
\medskip 

\noindent We call such a matrix $f(b)$ the {\it Burau matrix of $b$}. \\

Let $\pi$ be a permutation on $\{ 1, 2, \dots , m \}$. 
The {\it permutation matrix} $[p_{ij}]$ of $\pi$ is defined by 
\begin{align*}
p_{ij}=
\begin{cases}
1 & (\text{if } j= \pi(i)) \\
0 & ( \text{otherwise}).
\end{cases}
\end{align*}
For a braid $b$, we refer to the permutation matrix of the braid permutation of $b$ simply as the {\it permutation matrix of $b$}. 
The following proposition is well known (see, e.g., Lemma 2.2 in \cite{NS}).  

\medskip 
\begin{proposition}
For any $b \in B_m$, the matrix $\left.f(b)\right|_{t=1}$ is the permutation matrix of $b$. 
In particular, for any $b \in P_m$, we have $\left.f(b)\right|_{t=1}=I_m$. 
\label{prop-perm-matrix}
\end{proposition}
\medskip 

\begin{example}
For $\sigma_1 \sigma_2^{-1} \in B_3$ with braid permutation $\pi(1,2,3)=(3,1,2)$, 
\begin{align*}
\left.f(\sigma_1\sigma_2^{-1})\right|_{t=1}= \left. f(\sigma_1) \right|_{t=1} \left. f(\sigma_2^{-1}) \right|_{t=1} =
\begin{bmatrix}
0 & 1 & 0 \\
1 & 0 & 0 \\
0 & 0 & 1 
\end{bmatrix}
\begin{bmatrix}
1 & 0 & 0 \\
0 & 0 & 1 \\
0 & 1 & 0 
\end{bmatrix}
=
\begin{bmatrix}
0 & 0 & 1 \\
1 & 0 & 0 \\
0 & 1 & 0 
\end{bmatrix}.
\end{align*}
\end{example}

\subsection{Differentiated Burau representation}

For the Burau representation $f(b)=[f_{ij}(b;t)]$ of a braid $b \in B_m$, recall that the differentiated Burau matrix is defined as follows: 
\begin{align*}
\frac{d}{dt}f(b)=\left[ \frac{d}{dt} f_{i j} (b;t) \right] .
\end{align*}
We denote the differentiated Burau matrix evaluated at $t=1$ by 
\begin{align*}
\Phi (b) =\left. \frac{d}{dt} f(b) \right|_{t=1} \in M_m (\mathbb{Z}), 
\end{align*}
where $M_m( \mathbb{Z})$ denotes the set of $m \times m$ integer matrices. 

\medskip 
\begin{example}
For $b=\sigma_1\sigma_2^{-1} \in B_3$, we have 
\begin{align*}
\frac{d}{dt}f(b)= 
\begin{bmatrix}
-1 & 0 & -t^{-2} \\
1 & 0 & 0 \\
0 & 0 & t^{-2} 
\end{bmatrix}, \ \Phi (b)=
\begin{bmatrix}
-1 & 0 & -1 \\
1 & 0 & 0 \\
0 & 0 & 1 
\end{bmatrix}.
\end{align*}
\end{example}
\medskip

\noindent Let $M_m(\Lambda )$ denote the set of $m \times m$ matrices over $\Lambda$. 
In this section, we use the following product rule for matrix differentiation. 

\medskip 
\begin{lemma}
For $C, \ D \in M_m(\Lambda )$, we have 
\begin{align*}
\frac{d}{dt}(CD)= \left( \frac{d}{dt}C \right) D+C \left( \frac{d}{dt}D \right) .
\end{align*}
\label{lem-dt-Lei}
\end{lemma}

\begin{proof}
Let $C=[c_{ij}], D=[d_{ij}]$. 
By the Leibniz rule, the $(i,j)$ entry of $\frac{d}{dt}(CD)$ is 
\begin{align*}
\frac{d}{dt}\left( \sum_{k=1}^m c_{ik}d_{kj} \right) = \sum_{k=1}^m \left\{ \left( \frac{d}{dt} c_{ik}\right) d_{kj} + c_{ik} \left( \frac{d}{dt} d_{kj} \right) \right\} = \sum_{k=1}^m \left\{ \left( \frac{d}{dt} c_{ik} \right) d_{kj} \right\} + \sum_{k=1}^m \left\{ c_{ik} \left( \frac{d}{dt} d_{kj} \right) \right\}.
\end{align*}
This is the $(i,j)$ entry of $\left( \frac{d}{dt}C \right) D+C \left( \frac{d}{dt}D \right)$. 
\end{proof}
\medskip

\noindent The following proposition shows that $\Phi : P_m \to M_m( \mathbb{Z})$ is a group homomorphism. 

\medskip 
\begin{proposition}
For any $b, c \in P_m$, we have $\Phi (bc)=\Phi (b) + \Phi (c)$. 
\label{prop-Phi-hom}
\end{proposition}
\medskip 

\begin{proof}
Since $f$ is a group homomorphism, we have $f(bc)=f(b)f(c)$. 
By Lemma \ref{lem-dt-Lei}, we have 
\begin{align*}
\frac{d}{dt} f(bc)= \frac{d}{dt} \left( f(b)f(c) \right) = \left( \frac{d}{dt}f(b) \right)f(c)+f(b) \left( \frac{d}{dt}f(c) \right).
\end{align*}
Since $b, c \in P_m$, we have $\left. f(b) \right|_{t=1}=\left. f(c) \right|_{t=1}= I_m$ by Proposition \ref{prop-perm-matrix}.  
Therefore, 
\begin{align*}
\Phi (bc) = \left. \frac{d}{dt}f(bc) \right|_{t=1}=\left( \left.\frac{d}{dt}f(b) \right|_{t=1} \right) \left.f(c) \right|_{t=1}+\left. f(b) \right|_{t=1} \left( \left.\frac{d}{dt}f(c)\right|_{t=1} \right) 
= \Phi (b) I_m+I_m \Phi(c)=\Phi(b)+\Phi (c).
\end{align*}
\end{proof}

\begin{corollary}
For any $b \in P_m$, we have $\Phi(b^{-1})=-\Phi(b)$. 
\label{cor-Phi-1}
\end{corollary}
\medskip 

\begin{proof}
By Proposition \ref{prop-Phi-hom}, $\Phi(bb^{-1})=\Phi(b)+\Phi(b^{-1})=\Phi(id)=O$. 
Hence $\Phi(b^{-1})=-\Phi(b)$. 
\end{proof}
\medskip 

\noindent Let 
\begin{align*}
\boldsymbol{e}_1=
\begin{bmatrix}
1 \\
0 \\
\vdots \\
0
\end{bmatrix}, 
\boldsymbol{e}_2=
\begin{bmatrix}
0 \\
1 \\
\vdots \\
0
\end{bmatrix}, \dots ,
\boldsymbol{e}_m=
\begin{bmatrix}
0 \\
0 \\
\vdots \\
1
\end{bmatrix} \in \mathbb{Z}^m.
\end{align*}
For the hook generator $A_{i \ i+1}$, the following formula holds. 

\medskip 
\begin{lemma}
For the hook generator $A_{i \ i+1}= \sigma_i^2 \in P_m$ ($1 \leq i \leq m-1$), we have 
\begin{align*}
\Phi(A_{i \ i+1})=[ \boldsymbol{e}_i-\boldsymbol{e}_{i+1}][ \boldsymbol{e}_i-\boldsymbol{e}_{i+1}]^T.
\end{align*}
\label{lem-Aii}
\end{lemma}
\medskip 

\begin{example}
Let $A_{23}=\sigma_2^2 \in P_5$. Then 
\begin{align*}
f(\sigma_2^2)= 
\begin{bmatrix}
1 & 0 & 0 & 0 & 0 \\
0 & -t+1 & 1 & 0 & 0 \\
0 & t & 0 & 0 & 0 \\
0 & 0 & 0 & 1 & 0 \\
0 & 0 & 0 & 0 & 1 
\end{bmatrix}^2 = 
\begin{bmatrix}
1 & 0 & 0 & 0 & 0 \\
0 & t^2-t+1 & -t+1 & 0 & 0 \\
0 & -t^2+t & t & 0 & 0 \\
0 & 0 & 0 & 1 & 0 \\
0 & 0 & 0 & 0 & 1 
\end{bmatrix}, \ \frac{d}{dt}f(\sigma_2^2)=
\begin{bmatrix}
0 & 0 & 0 & 0 & 0 \\
0 & 2t-1 & -1 & 0 & 0 \\
0 & -2t+1 & 1 & 0 & 0 \\
0 & 0 & 0 & 0 & 0 \\
0 & 0 & 0 & 0 & 0 
\end{bmatrix}
\end{align*}
and 
\begin{align*}
\Phi (\sigma_2^2)=
\begin{bmatrix}
0 & 0 & 0 & 0 & 0 \\
0 & 1 & -1 & 0 & 0 \\
0 & -1 & 1 & 0 & 0 \\
0 & 0 & 0 & 0 & 0 \\
0 & 0 & 0 & 0 & 0 
\end{bmatrix}=
\begin{bmatrix}
0 \\
1 \\
-1 \\
0 \\
0
\end{bmatrix}
\begin{bmatrix}
0 \ 1   -1 \ 0 \ 0
\end{bmatrix}=
[ \boldsymbol{e}_2-\boldsymbol{e}_{3}][ \boldsymbol{e}_2-\boldsymbol{e}_{3}]^T .
\end{align*}
\end{example}
\medskip 

\noindent We now prove Lemma \ref{lem-Aii}.

\medskip
\begin{proof}[Proof of Lemma \ref{lem-Aii}.]
We have
\begin{align*}
f(\sigma_i^2) & = I_{i-1} \oplus 
\begin{bmatrix}
t^2-t+1 & -t+1 \\
-t^2+t & t
\end{bmatrix}
\oplus I_{m-i-1}, \\
\frac{d}{dt}f(\sigma_i^2) & = O_{i-1} \oplus 
\begin{bmatrix}
2t-1 & -1 \\
-2t+1 & 1
\end{bmatrix}
\oplus O_{m-i-1},
\end{align*}
and therefore 
\begin{align*}
\Phi(\sigma_i^2)= O_{i-1} \oplus 
\begin{bmatrix}
1 & -1 \\
-1 & 1
\end{bmatrix}
\oplus O_{m-i-1}=[ \boldsymbol{e}_i-\boldsymbol{e}_{i+1}][ \boldsymbol{e}_i-\boldsymbol{e}_{i+1}]^T.
\end{align*}
\end{proof}
\medskip 

\noindent In general, we have the following. 

\medskip 
\begin{lemma}
For the hook generator $A_{ij} \in P_m$ ($1 \leq i<j \leq m$), we have 
\begin{align*}
\Phi(A_{ij})=[ \boldsymbol{e}_i-\boldsymbol{e}_j][ \boldsymbol{e}_i-\boldsymbol{e}_j]^T.
\end{align*}
\label{lem-Phi-Aij}
\end{lemma}

\begin{proof}
Let $A_{ij}=q \sigma_i^2q^{-1}$, where $q=\sigma_{j-1} \sigma_{j-2} \dots \sigma_{i+1}$. 
Note that $q$ is non-pure when $i+1<j$.  
Let $F=F(t)=f(q)$, $G=G(t)=f(\sigma_i^2)$. 
Then $f(A_{ij})=F(t)G(t)F(t)^{-1}$. 
Applying the product rule as in Lemma \ref{lem-dt-Lei}, we obtain 
\begin{align}
\frac{d}{dt}f(A_{ij})=\frac{d}{dt}(FGF^{-1})=\frac{dF}{dt}GF^{-1}+F\frac{dG}{dt}F^{-1}+FG\frac{dF^{-1}}{dt}. 
\label{formula-1}
\end{align}
Here, by differentiating both sides of $F(t)F(t)^{-1}=I_m$, we obtain the equality 
\begin{align}
\frac{d}{dt}(F(t)^{-1})=-F(t)^{-1}\left( \frac{d}{dt}F(t) \right) F(t)^{-1}.
\label{formula-2}
\end{align}
Set $P=F(1)=\left.f(q)\right|_{t=1}$. 
Note that $P$ is the permutation matrix of $q$ by Proposition \ref{prop-perm-matrix}. 
Since $\sigma_i^2 \in P_m$, we have $G(1)=\left.f(\sigma_i^2)\right|_{t=1}=I_m$. 
By (\ref{formula-1}) and (\ref{formula-2}), 
\begin{align*}
\Phi(A_{ij})=\left.\frac{d}{dt}f(A_{ij})\right|_{t=1}=\left.\frac{dF}{dt}\right|_{t=1}P^{-1}+P \left.\frac{dG}{dt}\right|_{t=1}P^{-1}-P\left( P^{-1}\left.\frac{dF}{dt}\right|_{t=1}P^{-1}\right). 
\end{align*}
Hence 
\begin{align}
\Phi(A_{ij})=P\Phi(\sigma_i^2)P^{-1}. 
\label{formula-3}
\end{align}

Next, we show that $P[ \boldsymbol{e}_i-\boldsymbol{e}_{i+1}]=[ \boldsymbol{e}_i-\boldsymbol{e}_{j}]$. 
The permutation matrix of $q=\sigma_{j-1}\sigma_{j-2}\dots \sigma_{i+1}$ is $P=S_{j-1}S_{j-2}\dots S_{i+1}$, where $S_k=\left. f(\sigma_k)\right|_{t=1}$ is the permutation matrix of the transposition $(k \ k+1)$. 
Note that $P\boldsymbol{e}_i=\boldsymbol{e}_i$. 
On the other hand, $P\boldsymbol{e}_{i+1}=\boldsymbol{e}_j$ because 
\begin{align*}
P \boldsymbol{e}_{i+1} & = S_{j-1}S_{j-2} \dots S_{i+2}S_{i+1}\boldsymbol{e}_{i+1} \\
& =  S_{j-1}S_{j-2} \dots S_{i+2}\boldsymbol{e}_{i+2} \\
& = \dots = S_{j-1}\boldsymbol{e}_{j-1} =\boldsymbol{e}_{j}.
\end{align*}
Hence, 
\begin{align}
P[ \boldsymbol{e}_i-\boldsymbol{e}_{i+1}]=[ \boldsymbol{e}_i-\boldsymbol{e}_{j}].
\label{formula-4}
\end{align}
Since the permutation matrix $P$ is an orthogonal matrix and hence satisfies $P^{-1}=P^T$, we have the following by (\ref{formula-3}), (\ref{formula-4}), and Lemma \ref{lem-Aii}. 
\begin{align*}
\Phi (A_{ij}) & = P \Phi(\sigma_i^2)P^{-1} \\
& = P[ \boldsymbol{e}_i-\boldsymbol{e}_{i+1}][ \boldsymbol{e}_i-\boldsymbol{e}_{i+1}]^TP^T \\
& = \left( P[ \boldsymbol{e}_i-\boldsymbol{e}_{i+1}] \right) \left( P[ \boldsymbol{e}_i-\boldsymbol{e}_{i+1}] \right)^T \\
& = [ \boldsymbol{e}_i-\boldsymbol{e}_j]  [ \boldsymbol{e}_i-\boldsymbol{e}_j]^T.
\end{align*}
\end{proof}

\begin{example}
Let $A_{24}=\sigma_3\sigma_2^2\sigma_3^{-1} \in P_5$. Then 
\begin{align*}
f(A_{24})= 
\begin{bmatrix}
1 & 0 & 0 & 0 & 0 \\
0 & 1 & 0 & 0 & 0 \\
0 & 0 & 1-t & 1 & 0 \\
0 & 0 & t & 0 & 0 \\
0 & 0 & 0 & 0 & 1 
\end{bmatrix}
\begin{bmatrix}
1 & 0 & 0 & 0 & 0 \\
0 & 1-t & 1 & 0 & 0 \\
0 & t & 0 & 0 & 0 \\
0 & 0 & 0 & 1 & 0 \\
0 & 0 & 0 & 0 & 1 
\end{bmatrix}^2 
\begin{bmatrix}
1 & 0 & 0 & 0 & 0 \\
0 & 1 & 0 & 0 & 0 \\
0 & 0 & 0 & t^{-1} & 0 \\
0 & 0 & 1 & 1-t^{-1} & 0 \\
0 & 0 & 0 & 0 & 1 
\end{bmatrix}=
\begin{bmatrix}
1 & 0 & 0 & 0 & 0 \\
0 & t^2-t+1 & 0 & -1+t^{-1} & 0 \\
0 & t^3-2t^2+t & 1 & -t+2-t^{-1} & 0 \\
0 & -t^3+t^2 & 0 & t & 0 \\
0 & 0 & 0 & 0 & 1 
\end{bmatrix},
\end{align*}
\begin{align*}
\frac{d}{dt}f(A_{24})=
\begin{bmatrix}
0 & 0 & 0 & 0 & 0 \\
0 & 2t-1 & 0 & -t^{-2} & 0 \\
0 & 3t^2-4t+1 & 0 & -1+t^{-2} & 0 \\
0 & -3t^2+2t & 0 & 1 & 0 \\
0 & 0 & 0 & 0 & 0 
\end{bmatrix}, \Phi(A_{24})=
\begin{bmatrix}
0 & 0 & 0 & 0 & 0 \\
0 & 1 & 0 & -1 & 0 \\
0 & 0 & 0 & 0 & 0 \\
0 & -1 & 0 & 1 & 0 \\
0 & 0 & 0 & 0 & 0 
\end{bmatrix}=
\begin{bmatrix}
0 \\
1 \\
0 \\
-1 \\
0
\end{bmatrix}
\begin{bmatrix}
0 \ 1 \ 0  -1 \ 0
\end{bmatrix}.
\end{align*}
\label{ex-Phi-A24}
\end{example}

\section{Crossing matrix}
\label{section-CM}

In this section, we review the crossing matrix of braids defined in \cite{Bu} (see also \cite{Gu}). 
Let $\beta$ be a diagram of a braid $b \in B_m$. 
We call a crossing of $\beta$ that corresponds to $\sigma_i$ (resp. $\sigma_i^{-1}$) for some $i$ a {\it positive crossing} (resp. {\it negative crossing}). 
The {\it crossing matrix of $\beta$}, denoted by $C(\beta)$, is an $m \times m$ matrix with zero diagonal such that the $(i,j)$ entry denotes the number of positive crossings minus the number of negative crossings at which the $i^{th}$ strand is over the $j^{th}$ strand in $\beta$. 
The matrix $C(\beta)$ is invariant under the braid relations and therefore depends only on $b$. 
Then the {\it crossing matrix $C(b)$ of a braid $b\in B_m$} is defined as $C(b)=C(\beta)$ for any diagram $\beta$ of $b$. 

\medskip 
\begin{example}
For $b=\sigma_1 \sigma_2^{-1} \in B_3$, we have 
\begin{align*}
C(b)=
\begin{bmatrix}
0 & 0 & -1 \\
1 & 0 & 0 \\
0 & 0 & 0
\end{bmatrix}.
\end{align*}
\end{example}
\medskip 

\noindent Let $M_m( \mathbb{Z})$ denote the set of $m \times m$ integer matrices. 
Note that $M_m(\mathbb{Z})$ is a group under addition. 
We define a map $C:B_m \to M_m( \mathbb{Z})$ by $b \mapsto C(b)$. 
For pure braids, the following is shown in \cite{Bu}. 

\medskip 
\begin{lemma}[\cite{Bu}]
Let $b, c \in P_m$. 
\begin{itemize}
\item[(1)] The crossing matrix $C(b)$ is symmetric. 
\item[(2)] We have $C(bc)=C(b)+C(c)$. 
\end{itemize}
\label{lem-CM-properties}
\end{lemma}
\medskip 

\noindent Now we define the Laplacian of a matrix. 

\medskip
\begin{definition}
Let $A=[a_{ij}]$ be an $m \times m$ symmetric matrix with zero diagonal. 
\begin{itemize}
\item[(1)] The following diagonal matrix $D_A \in M_m(\mathbb{Z})$ is called the {\it degree matrix of $A$}. 
\begin{align*}
D_A= 
\begin{bmatrix}
\sum_{j\neq 1}a_{1 j}  \\
 & \sum_{j\neq 2}a_{2 j}          &        & \text{\Huge{0}}   \\
 &                 & \ddots                     \\
  & \text{\Huge{0}} &        & \ddots            \\
&                 &        &           & \sum_{j\neq m}a_{m j}
\end{bmatrix}
\end{align*}
\item[(2)] The matrix $L_A=D_A-A \in M_m(\mathbb{Z})$ is called the {\it Laplacian of $A$}. 
\end{itemize}
\end{definition}
\medskip

\begin{example}
\begin{align*}
\text{\small When } A=
\begin{bmatrix}
0 & 2 & -1 \\
2 & 0 & 3 \\
-1 & 3 & 0
\end{bmatrix}, \text{ we have }
D_A=
\begin{bmatrix}
1 & 0 & 0 \\
0 & 5 & 0 \\
0 & 0 & 2
\end{bmatrix}, \ 
L_A= D_A -A =
\begin{bmatrix}
1 & -2 & 1 \\
-2 & 5 & -3 \\
1 & -3 & 2
\end{bmatrix}. 
\end{align*}
\end{example}
\medskip 

\noindent The following proposition shows that the map $L:P_m \to M_m( \mathbb{Z})$ with $b \mapsto L_{C(b)}$ is a group homomorphism. 

\medskip 
\begin{proposition}
For any $b, c \in P_m$, we have $L_{C(bc)}=L_{C(b)}+L_{C(c)}$. 
\label{prop-LC-hom}
\end{proposition}
\medskip 

\begin{proof}
By Lemma \ref{lem-CM-properties} (2), we have $C(bc)=C(b)+C(c)$ for any $b, c \in P_m$. 
For the degree matrix, we have $D_{C(bc)}=D_{C(b)+C(c)}=D_{C(b)}+D_{C(c)}$. 
Hence, 
\begin{align*}
L_{C(bc)} & = D_{C(bc)}-C(bc) \\
& =( D_{C(b)}+D_{C(c)})-(C(b)+C(c)) \\
& =(D_{C(b)}-C(b))+(D_{C(c)}-C(c)) \\
& = L_{C(b)}+L_{C(c)}.
\end{align*}
\end{proof}

\begin{corollary}
For any $b \in P_m$, we have $L_{C(b^{-1})}=-L_{C(b)}$. 
\label{cor-LC-1}
\end{corollary}
\medskip 

\begin{proof}
By Proposition \ref{prop-LC-hom}, $L_{C(bb^{-1})}=L_{C(b)}+L_{C(b^{-1})}=L_{C(id)}=O$. 
Hence, $L_{C(b^{-1})}=-L_{C(b)}$. 
\end{proof}

\medskip 
\begin{lemma}
For the hook generator $A_{ij} \in P_m$ ($1 \leq i<j \leq m$), we have 
\begin{align*}
L_{C(A_{ij})}=  [ \boldsymbol{e}_i-\boldsymbol{e}_j]  [ \boldsymbol{e}_i-\boldsymbol{e}_j]^T.
\end{align*}
\label{lem-CM-e}
\end{lemma}
\medskip

\begin{example}
When $b=A_{24} \in P_5$, 
\begin{align*}
C(b)=
\begin{bmatrix}
0 & 0 & 0 & 0 & 0 \\
0 & 0 & 0 & 1 & 0 \\
0 & 0 & 0 & 0 & 0 \\
0 & 1 & 0 & 0 & 0 \\
0 & 0 & 0 & 0 & 0 
\end{bmatrix}, \ D_{C(b)}=
\begin{bmatrix}
0 & 0 & 0 & 0 & 0 \\
0 & 1 & 0 & 0 & 0 \\
0 & 0 & 0 & 0 & 0 \\
0 & 0 & 0 & 1 & 0 \\
0 & 0 & 0 & 0 & 0 
\end{bmatrix}, \text{ and }  L_{C(b)}=
\begin{bmatrix}
0 & 0 & 0 & 0 & 0 \\
0 & 1 & 0 & -1 & 0 \\
0 & 0 & 0 & 0 & 0 \\
0 & -1 & 0 & 1 & 0 \\
0 & 0 & 0 & 0 & 0 
\end{bmatrix} = 
\begin{bmatrix}
0 \\
1 \\
0 \\
-1 \\
0
\end{bmatrix} 
\begin{bmatrix}
0 \ 1 \ 0 -1 \ 0
\end{bmatrix}.
\end{align*}
\end{example}
\medskip 

\begin{proof}[Proof of Lemma \ref{lem-CM-e}.] 
In the diagram of $A_{ij}$ depicted in Figure \ref{fig-hook}, the $i^{th}$ and $j^{th}$ strands have two positive crossings at which the $i^{th}$ strand is over at one crossing and the $j^{th}$ strand is over at the other crossing. 
The $k^{th}$ strand ($i+1 \leq k \leq j-1$) and the $j^{th}$ strand have two crossings of opposite signs at which the $j^{th}$ strand is over. 
Other pairs have no mutual crossings. 
Hence, the crossing matrix $C(A_{ij})$ has $1$ at the $(i,j)$ and $(j,i)$ entries, and has 0 for other entries. 
Then, $D_{C(A_{ij})}$ has 1 at the $(i,i)$ and $(j,j)$ entries. 
Therefore, $L_{C(A_{ij})}$ has $-1$ at the $(i,j)$ and $(j,i)$ entries, has $1$ at the $(i,i)$, $(j,j)$ entries, and has 0 for other entries. 
This implies $L_{C(A_{ij})} = [ \boldsymbol{e}_i-\boldsymbol{e}_j] [ \boldsymbol{e}_i-\boldsymbol{e}_j]^T$. 
\end{proof}
\medskip

\section{Proof of the main theorem}
\label{section-pf}

In this section, we prove Theorem \ref{thm-main}. 
It suffices to prove the following lemma for the hook generator.  

\medskip 
\begin{lemma}
For each hook generator $A_{ij} \in P_m$ ($1 \leq i<j \leq m$), we have $\Phi (A_{ij}) =L_{C(A_{ij})}$. 
\label{lem-Phi-LC}
\end{lemma}
\medskip 

\begin{proof}
By Lemmas \ref{lem-CM-e} and \ref{lem-Phi-Aij}, we obtain $\Phi(A_{ij})=L_{C(A_{ij})}= [ \boldsymbol{e}_i-\boldsymbol{e}_j]  [ \boldsymbol{e}_i-\boldsymbol{e}_j]^T$. 
\end{proof}
\medskip

\noindent Now we prove Theorem \ref{thm-main}. 

\medskip 
\begin{proof}[Proof of Theorem \ref{thm-main}.]
Let $b=A_{i_1j_1}^{\varepsilon_1}A_{i_2j_2}^{\varepsilon_2} \dots A_{i_rj_r}^{\varepsilon_r} \in P_m$ ($1 \leq i_k < j_k \leq m$, $\varepsilon_k \in \{ \pm 1\}$). 
By Proposition \ref{prop-Phi-hom} and Corollary \ref{cor-Phi-1}, we have 
$$\Phi(b)= \sum_{k=1}^r \varepsilon_k \Phi(A_{i_kj_k}).$$
By Proposition \ref{prop-LC-hom} and Corollary \ref{cor-LC-1}, we have 
$$L_{C(b)}= \sum_{k=1}^r \varepsilon_k L_{C(A_{i_kj_k})}.$$
By Lemma \ref{lem-Phi-LC}, 
$$\Phi(b)=\sum_{k=1}^r \varepsilon_k \Phi(A_{i_kj_k})=\sum_{k=1}^r \varepsilon_k L_{C(A_{i_kj_k})}=L_{C(b)}.$$
Therefore, $\Phi(b)=L_{C(b)}$ holds. 
\end{proof}
\medskip 

\begin{proof}[Proof of Corollary \ref{cor-cij}.]
It follows immediately from Theorem \ref{thm-main}. 
Note that the crossing matrix $C(b)$ is recovered from its Laplacian $L_{C(b)}$ by replacing each diagonal entry with zero and multiplying the remaining entries by $-1$. 
\end{proof}
\medskip 

\begin{proof}[Proof of Corollary \ref{cor-symmetric}.]
It follows from Theorem \ref{thm-main} and Lemma \ref{lem-CM-properties} (1). 
Note that if $C(b)$ is symmetric, then $L_{C(b)}$ is also symmetric. 
\end{proof}
\medskip

\section{Conjugacy invariants}
\label{section-conjugacy}

Two braids $b, b' \in B_m$ are {\it conjugate} if $b'=c^{-1}bc$ for some $c \in B_m$. 
As we have seen, Corollary \ref{cor-cij} shows that the crossing matrix of a pure braid can be recovered from the Burau matrix. 
In this section, we compare methods based on the Burau matrix and the crossing matrix for proving that two pure braids are not conjugate.

\subsection{Conjugacy invariants derived from the Burau representation}
\label{subsection-C-BR}

The Burau representation gives several useful conjugacy invariants\footnote{
As observed in \cite{JSB}, the characteristic polynomial of the reduced Burau matrix does not determine completely the conjugacy class of $b \in B_n$ (Corollary 3.11.2 in \cite{JSB}). 
}
that are derived from the equation $f(c^{-1}bc)=f(c)^{-1}f(b)f(c)$. 
Two matrices $A, A' \in M_m(\Lambda)$ are {\it similar} if $A'=P^{-1}AP$ for some $P\in GL_m(\Lambda)$. 
Then we have the following. 

\medskip 
\begin{proposition}
If $b$ and $b' \in B_m$ are conjugate, then $f(b)$ and $f(b')$ are similar. 
\end{proposition}
\medskip 

\noindent In general, determining whether two matrices are similar is not easy.  Moreover, the image of the Burau representation is not completely understood\footnote{
As for the crossing matrix, crossing matrices of braids, pure braids are completely characterized in \cite{Bu}. 
The characterization for positive pure braids remains an open problem. 
It was determined up to degree six by Y. Ozawa and the authors (see \cite{YAY, AY-5}). 
}. 
The following corollary is useful in practice. 

\medskip 
\begin{corollary}
If $b$ and $b' \in B_m$ are conjugate, then $f(b)$ and $f(b')$ have the same trace, determinant, and characteristic polynomial. 
\end{corollary}

\subsection{Conjugacy invariants derived from the crossing matrix}
\label{subsection-C-CM}

Two matrices $A, A' \in M_m(\mathbb{Z})$ are {\it permutation equivalent} if there exists a permutation matrix $T$ such that $A'=TAT^T$. 
For pure braids, we have the following. 

\medskip 
\begin{proposition}[\cite{AS}]
If $b$ and $b' \in P_m$ are conjugate, then $C(b)$ and $C(b')$ are permutation equivalent. 
\end{proposition}
\medskip 

\noindent Note that permutation equivalence can be checked by examining at most $m!$ permutations. 
In \cite{AS, AY-H}, conjugacy invariants for pure braids are also obtained from the crossing matrix.

\begin{corollary}[\cite{AS, AY-H}]
If $b$ and $b' \in P_m$ are conjugate, then $C(b)$ and $C(b')$ have the same rank, determinant, and characteristic polynomial. 
They also have the same multiset of entries. 
\end{corollary}
\medskip 

\noindent For non-pure braids, the following was shown in \cite{AS}. 

\medskip 
\begin{proposition}[\cite{AS}]
Let $r$, $r'$ be the orders of the braid permutations of $b$, $b' \in B_m$, respectively. 
If $b$ and $b' \in B_m$ are conjugate, then $C(b^r)$ and $C((b')^{r'})$ are permutation equivalent. 
\end{proposition}

\subsection{Examples}

In this subsection, we compare the effectiveness of conjugacy invariants derived from the Burau matrix and the crossing matrix. 

\medskip 
\begin{example}
Let $b_3=(\sigma_1 \sigma_2^{-1})^3 \in P_3$. 
Then $C(b_3)=O_3=C(id_3)$. 
In this case, $b_3$ and $id_3 \in P_3$ have the same crossing matrix and therefore we cannot determine the conjugacy by the crossing matrix for this pair. 
For the Burau representation, we have 
\begin{align*}
f(b_3)= (f(\sigma_1 \sigma_2^{-1}))^3=
\begin{bmatrix}
-t^3+3t^2-3t+2 & -1+2t^{-1}-t^{-2} & t-3+4t^{-1}-3t^{-2}+t^{-3} \\
t^3-2t^2+t & 1 & -t+2-t^{-1} \\
-t^2+2t-1 & 1-2t^{-1}+t^{-2} & 2-3t^{-1}+3t^{-2}-t^{-3}
\end{bmatrix}.
\end{align*}
The trace of $f(b_3)$ is $-t^3+3t^2-3t+5-3t^{-1}+3t^{-2}-t^{-3}$. 
The identity braid $id_3$ has the Burau matrix $f(id_3)=I_3$ with trace $3$. 
Since their Burau matrices have different traces, we can conclude that $b_3$ and $id_3$ are not conjugate. 
\end{example}
\medskip 

\noindent Although the crossing matrix of a pure braid is determined by the Burau matrix by Corollary \ref{cor-cij}, conjugacy invariants based on the crossing matrix may be easier to compute and can distinguish some non-conjugate pairs that have the same characteristic polynomial and hence the same trace and determinant of the Burau matrix.

\medskip 
\begin{example}
Let $b_1=\sigma_1^6 \sigma_3^2$, $b_2=\sigma_1^2\sigma_2\sigma_1^2\sigma_2^{-1}\sigma_1^2\sigma_2\sigma_1^2\sigma_2^{-1} \in P_4$. 
The Burau representations are 
\begin{align*}
f(b_1)=
\begin{bmatrix}
t^6-t^5+t^4-t^3+t^2-t+1 & -t^5+t^4-t^3+t^2-t+1 & 0 & 0 \\
-t^6+t^5-t^4+t^3-t^2+t & t^5-t^4+t^3-t^2+t & 0 & 0 \\
0 & 0 & t^2-t+1 & -t+1 \\
0 & 0 & -t^2+t & t
\end{bmatrix}, 
\end{align*}
\begin{align*}
f(b_2)=
\begin{bmatrix}
t^6-t^4+t^3-t+1 & -t^4+t^3-t+1 & -t^4+t^3-t+1 & 0 \\
-t^5+t^4-t^2+t & t^3-t^2+t & t^3-2t^2+t & 0 \\
-t^6+t^5-t^3+t^2 & t^4-2t^3+t^2 & t^4-2t^3+2t^2 & 0 \\
0 & 0 & 0 & 1
\end{bmatrix}.
\end{align*}
The characteristic polynomials of $f(b_1)$ and $f(b_2)$ are both $(\lambda -1)^2(\lambda -t^6)(\lambda -t^2)$. 
On the other hand, the crossing matrices are 
\begin{align*}
C(b_1)=
\begin{bmatrix}
0 & 3 & 0 & 0 \\
3 & 0 & 0 & 0 \\
0 & 0 & 0 & 1 \\
0 & 0 & 1 & 0 
\end{bmatrix}, 
C(b_2)=
\begin{bmatrix}
0 & 2 & 2 & 0 \\
2 & 0 & 0 & 0 \\
2 & 0 & 0 & 0 \\
0 & 0 & 0 & 0 
\end{bmatrix}. 
\end{align*}
Since $C(b_1)$ and $C(b_2)$ have different multisets of entries, we can conclude that they are not permutation equivalent. 
Hence $b_1$ and $b_2$ are not conjugate. 
Note that the characteristic polynomials of $C(b_1)$ and $C(b_2)$ are $(\lambda^2-9)(\lambda^2-1)$ and $\lambda^2(\lambda^2-8)$, respectively. 
\end{example}

\section*{Acknowledgments}
This work was partially supported by the JSPS KAKENHI Grant Number JP21K03263.

\end{document}